\documentclass[12pt]{amsart}
\usepackage{amsfonts,amssymb,amscd,amsmath,enumerate,color, xypic}
\def\NZQ{\mathbb}               

\def\ZZ{{\NZQ Z}}
\def\RR{{\NZQ R}}

\def\frk{\mathfrak}               

\def\Phi{{\frk N}}
\def\ab{{\bold a}}
\def\bb{{\bold b}}
\def\eb{{\bold e}}

\def\vb{{\bold v}}
\def\wb{{\bold w}}
\def\xb{{\bold x}}
\def\yb{{\bold y}}

\def\opn#1#2{\def#1{\operatorname{#2}}} 
\opn\chara{char} 
\opn\length{\ell} 
\opn\pd{pd} 
\opn\rk{rk}
\opn\projdim{proj\,dim} 
\opn\injdim{inj\,dim} 
\opn\rank{rank}
\opn\depth{depth} 
\opn\grade{grade} 
\opn\height{height}
\opn\embdim{emb\,dim} 
\opn\codim{codim}

\opn\Tr{Tr} 
\opn\bigrank{big\,rank}
\opn\superheight{superheight}
\opn\lcm{lcm}
\opn\trdeg{tr\,deg}
\opn\reg{reg} 
\opn\lreg{lreg} 
\opn\ini{in} 
\opn\lpd{lpd}
\opn\size{size}
\opn\mult{mult}
\opn\dist{dist}
\opn\cone{cone}
\opn\lex{lex}
\opn\rev{rev}
\opn\div{div} \opn\Div{Div} \opn\cl{cl} \opn\Cl{Cl}
\opn\Spec{Spec} \opn\Supp{Supp} \opn\supp{supp} \opn\Sing{Sing}
\opn\Ass{Ass} \opn\Min{Min}
\opn\Ann{Ann} \opn\Rad{Rad} \opn\Soc{Soc}
\opn\Syz{Syz} \opn\Im{Im} \opn\Ker{Ker} \opn\Coker{Coker}
\opn\Am{Am} \opn\Hom{Hom} \opn\Tor{Tor} \opn\Ext{Ext}
\opn\End{End} \opn\Aut{Aut} \opn\id{id} \opn\ini{in}

\opn\nat{nat}
\opn\pff{pf}
\opn\Pf{Pf} \opn\GL{GL} \opn\SL{SL} \opn\mod{mod} \opn\ord{ord}
\opn\Gin{Gin}
\opn\Hilb{Hilb}\opn\adeg{adeg}\opn\std{std}\opn\ip{infpt}
\opn\Pol{Pol}
\opn\sat{sat}
\opn\Var{Var}
\opn\Gen{Gen}

\opn\aff{aff} \opn\con{conv} \opn\relint{relint} \opn\st{st}
\opn\lk{lk} \opn\cn{cn} \opn\core{core} \opn\vol{vol}
\opn\link{link} \opn\star{star}
\opn\gr{gr}

\def\Ec{{\mathcal E}}
\def\Mc{{\mathcal M}}
\def\Hc{{\mathcal H}}

\def\Gc{{\mathcal G}}

\def\Pc{{\mathcal P}}
\def\Qc{{\mathcal Q}}
\def\Rc{{\mathcal R}}

\def\pot#1#2{#1[\kern-0.28ex[#2]\kern-0.28ex]}

\opn\dirlim{\underrightarrow{\lim}}
\opn\inivlim{\underleftarrow{\lim}}
\let\to=\rightarrow

\def\Implies{\ifmmode\Longrightarrow \else
	\unskip${}\Longrightarrow{}$\ignorespaces\fi}
\def\implies{\ifmmode\Rightarrow \else
	\unskip${}\Rightarrow{}$\ignorespaces\fi}
\def\iff{\ifmmode\Longleftrightarrow \else
	\unskip${}\Longleftrightarrow{}$\ignorespaces\fi}

\let\:=\colon
\newtheorem{Theorem}{Theorem}[section]
\newtheorem{Lemma}[Theorem]{Lemma}
\newtheorem{Corollary}[Theorem]{Corollary}

\newtheorem{Example}[Theorem]{Example}

\newtheorem*{acknowledgement}{Acknowledgment}
\let\epsilon\varepsilon
\let\phi=\varphi
\let\kappa=\varkappa
\def\qed{\ifhmode\textqed\fi
	\ifmmode\ifinner\quad\qedsymbol\else\dispqed\fi\fi}
\def\textqed{\unskip\nobreak\penalty50
	\hskip2em\hbox{}\nobreak\hfil\qedsymbol
	\parfillskip=0pt \finalhyphendemerits=0}
\def\dispqed{\rlap{\qquad\qedsymbol}}

\opn\dis{dis}
\opn\height{height}
\opn\dist{dist}
\def\pnt{{\raise0.5mm\hbox{\large\bf.}}}

\opn\Lex{Lex}
\opn\conv{conv}

\begin{document}
\title{Reflexive polytopes arising from perfect graphs}
\author[T. Hibi]{Takayuki Hibi}
\address[Takayuki Hibi]{Department of Pure and Applied Mathematics,
	Graduate School of Information Science and Technology,
	Osaka University,
	Suita, Osaka 565-0871, Japan}
\email{hibi@math.sci.osaka-u.ac.jp}
\author[A. Tsuchiya]{Akiyoshi Tsuchiya}
\address[Akiyoshi Tsuchiya]{Department of Pure and Applied Mathematics,
	Graduate School of Information Science and Technology,
	Osaka University,
	Suita, Osaka 565-0871, Japan}
\email{a-tsuchiya@ist.osaka-u.ac.jp}

\subjclass[2010]{13P10, 52B20}
\date{}
\keywords{reflexive polytope, integer decomposition property, perfect graph, Ehrhart $\delta$-polynomial, Gr\"{o}bner basis}

\begin{abstract}
Reflexive polytopes form one of the distinguished classes of lattice polytopes.
Especially reflexive polytopes which possess the integer decomposition property
are of interest.  In the present paper, by virtue of the algebraic technique on
Gr\"onbner bases, a new class of reflexive polytopes which possess the integer
decomposition property and which arise from perfect graphs will be presented.
Furthermore, the Ehrhart $\delta$-polynomials of these polytopes will be studied.
\end{abstract} 

\maketitle
\section*{Background}
The reflexive polytope is one of the keywords belonging to the current trends 
on the research of convex polytopes.  In fact, many authors have studied 
reflexive polytopes from viewpoints of combinatorics, commutative algebra
and algebraic geometry.
It is known that reflexive polytopes correspond to Gorenstein toric Fano varieties, and they are related with
mirror symmetry (see, e.g., \cite{mirror,Cox}).
In each dimension there exist only finitely many reflexive polytopes 
up to unimodular equivalence (\cite{Lag}) 
and all of them are known up to dimension $4$ (\cite{Kre}).
Moreover, every lattice polytope is a face of a reflexive polytope (\cite{refdim}).
We are especially interested in reflexive polytopes with the integer
decomposition property, where the integer decomposition property is
particularly important in the theory and application of integer
programing \cite[\S 22.10]{integer}.
A lattice polytope which possesses the integer decomposition property is  normal and very ample.
These properties play important roles in algebraic  geometry.
Hence to find new classes of reflexive polytopes with the integer decomposition property is one of the most important problem.
For example, the following classes of reflexive polytopes with the integer decomposition property are known:
\begin{itemize}
	\item Centrally symmetric configurations (\cite{cent}).
	\item Reflexive polytopes arising from the order polytopes and the chain polytopes of finite partially ordered sets (\cite{twin,orderchain,volOC,omega}).
	\item Reflexive polytopes arising from the stable sets polytopes of perfect graphs (\cite{harmony}).
\end{itemize}

Following the previous work \cite{harmony} 
the present paper discusses a new class of reflexive polytopes which possess the integer decomposition property and which arise from perfect graphs.

\begin{acknowledgement}{\rm
		The authors would like to thank anonymous referees for reading the manuscript carefully.
		The second author is partially supported by Grant-in-Aid for JSPS Fellows 16J01549.
	}
\end{acknowledgement}

\section{Perfect graphs and reflexive polytopes}
Recall that a {\em lattice polytope} is a convex polytope 
all of whose vertices have integer coordinates.
We say that a lattice polytope $\Pc \subset \RR^d$ of dimension $d$ is 
{\em reflexive} if the origin of $\RR^d$ belongs to the interior of $\Pc$ and 
if the dual polytope 
\[
\Pc^{\vee} = \{ {\bf x} \in \RR^{d} \, : \, \langle {\bf x}, {\bf y} \rangle \le 1,
\, \forall {\bf y} \in \Pc \}
\]
is again a lattice polytope.  Here $\langle {\bf x}, {\bf y} \rangle$
is the canonical inner product of $\RR^d$. 
A lattice polytope $\Pc \subset \RR^d$ possesses the 
{\em integer decomposition property} if, for each integer $n \geq 1$ 
and for each $\ab \in n\Pc \cap \ZZ^d$,
where $n \Pc = \{ n \ab \, : \, \ab \in \Pc \}$,  
there exist $\ab_1,\ldots,\ab_n$ belonging to $\Pc \cap \ZZ^d$ with $\ab=\ab_1+\cdots+\ab_n$.

Let $G$ be a finite simple graph on the vertex set $[d] = \{1, \ldots, d \}$
and $E(G)$ the set of edges of $G$.
(A finite graph $G$ is called simple if $G$ possesses no loop and no multiple edge.)
A subset $W \subset [d]$ is called {\em stable}  
if, for all $i$ and $j$ belonging to $W$ with $i \neq j$,
one has $\{i,j\} \not\in E(G)$.
We remark that a stable set is often called an {\em independent set}.
A {\em clique} of $G$ is a subset $W \subset [d]$
which is a stable set of the complementary graph $\overline{G}$ of $G$.
The {\em chromatic number} of $G$ is the smallest integer $t \geq 1$ for which
there exist stable set $W_{1}, \ldots, W_{t}$ of $G$ with
$[d] = W_{1} \cup \cdots \cup W_{t}$.
A finite simple graph $G$ is said to be {\em perfect} 
(\cite{sptheorem}) if, 
for any induced subgraph $H$ of $G$
including $G$ itself,
the chromatic number of $H$ is
equal to the maximal cardinality of cliques of $H$.
The complementary graph of a perfect graph is perfect (\cite{sptheorem}).

Let ${\bf e}_1, \ldots, \eb_{d}$ denote 
the standard coordinate unit vectors of ${\RR}^d$.
Given a subset $W \subset [d]$,
we may associate $\rho(W) = \sum_{j \in W} {\eb}_j \in {\RR}^d$.
In particular $\rho(\emptyset)$ is the origin ${\bf 0}_d$ of $\RR^d$.
Let $S(G)$ denote the set of stable sets of $G$.
One has $\emptyset \in S(G)$ and $\{ i \} \in S(G)$
for each $i \in [d]$.
The {\em stable set polytope} $\Qc_G \subset \RR^{d}$
of $G$ is the $(0, 1)$-polytope which is the convex hull of 
$\{\rho(W) \, : \, W \in S(G) \}$ in $\RR^d$.
One has $\dim \Qc_G = d$.
 
Given lattice polytopes $\Pc \subset \RR^d$ and $\Qc \subset \RR^d$ of dimension $d$, 
we introduce the lattice polytopes $\Gamma(\Pc,\Qc) \subset \RR^{d}$ and
$\Omega(\Pc,\Qc) \subset \RR^{d+1}$ with
\[
\Gamma(\Pc,\Qc)=\textnormal{conv}\{\Pc\cup (-\Qc)\}, 
\]\[
\Omega(\Pc,\Qc)=\textnormal{conv}\{(\Pc\times\{1\}) \cup (-\Qc\times\{-1\})\}.
\]

It is natural to ask when $\Gamma(\Qc_{G_1},\Qc_{G_2})$ is reflexive and
when $\Omega(\Qc_{G_1},\Qc_{G_2})$ is reflexive, where $G_1$ and $G_2$ are
finite simple graphs on $[d]$.  The former is completely solved 
by \cite{harmony} and the later is studied in this paper.  In fact,

\begin{Theorem}
\label{Gorenstein}
Let $G_1$ and $G_2$ be finite simple graphs on $[d]$.
\begin{enumerate}
\item[{\rm (a)}]
{\rm (\cite{harmony})}
The following conditions are equivalent:
\begin{itemize}
	\item[(i)] The lattice polytope $\Gamma(\Qc_{G_1},\Qc_{G_2})$ is reflexive;
	\item [(ii)]The lattice polytope $\Gamma(\Qc_{G_1},\Qc_{G_2})$ is reflexive and possesses the integer decomposition property;
	\item [(iii)]Both $G_1$ and $G_2$ are perfect.
\end{itemize}
\item[{\rm (b)}]
The following conditions are equivalent:
\begin{itemize}
	\item[(i)] The lattice polytope $\Omega(\Qc_{G_1},\Qc_{G_2})$ possesses the integer decomposition property;
	\item [(ii)]The lattice polytope $\Omega(\Qc_{G_1},\Qc_{G_2})$ is reflexive and possesses the integer decomposition property;
	\item [(iii)]Both $G_1$ and $G_2$ are perfect.
\end{itemize}
\end{enumerate} 
\end{Theorem}
A proof of part (b) will be achieved in Section $2$.
It would, of course, be of interest to find a complete characterization
for $\Omega(\Qc_{G_1},\Qc_{G_2})$ to be reflexive.
For a finite simple graph $G$ on $[d]$, $\Omega(\Qc_G,\Qc_G)$ is called the \textit{Hansen polytope} $\Hc(G)$ of $G$.
This polytope possesses nice properties (e.g., centrally symmetric and $2$-level) and is studied in \cite{Hansen2,Hansen1}.
Especially, in \cite{Hansen2}, it is shown that if $G$ is perfect, then $\Hc(G)$ is reflexive.
Theorem \ref{Gorenstein} (b) says that $G$ is perfect if and only if the Hansen polytope $\Hc(G)$ possesses
the integer decomposition property.

If $G_1$ and $G_2$ are not perfect, $\Gamma(\Qc_{G_1},\Qc_{G_2})$
may not possess the integer decomposition property (Examples \ref{e1} and \ref{e2}).  
Furthermore, if $G_1$ and $G_2$ are not perfect,  
$\Omega(\Qc_{G_1},\Qc_{G_2})$
may not be reflexive (Examples \ref{e2} and \ref{e3}).

We now turn to the discussion of Ehrhart $\delta$-polynomials of
$\Gamma(\Qc_{G_1},\Qc_{G_2})$ and $\Omega(\Qc_{G_1},\Qc_{G_2})$.
Let, in general, $\Pc \subset \RR^d$ be a lattice polytope of dimension $d$.
The {\em Ehrhart $\delta$-polynomial} of $\Pc$ is the polynomial
\[
\delta(\Pc, \lambda) = 
(1 - \lambda)^{d+1} \left[ \,
1 + \sum_{n=1}^{\infty} \mid n\Pc \cap \ZZ^d \mid \, \lambda^n \, \right]
\] 
in $\lambda$.
Each coefficient of $\delta(\Pc, \lambda)$ is a nonnegative integer 
and the degree of $\delta(\Pc, \lambda)$ is at most $d$.
In addition $\delta(\Pc, 1)$ coincides with the normalized volume of $\Pc$, denoted by ${\rm vol}(\Pc)$.
Refer the reader to \cite[Part II]{HibiRedBook} for the detailed information 
about Ehrhart $\delta$-polynomials. 
Moreover, in \cite{Athanasiadis}, the Ehrhart theory for stable set polytopes is studied.

The {\em suspension} of a finite simple graph $G$ on $[d]$ is the finite
simple graph $\widehat{G}$ on $[d+1]$ with 
$E(\widehat{G}) = E(G) \cup \{ \{i, d+1\} \, : \, i \in [d] \}$.
We obtain the following theorem.

\begin{Theorem}
\label{delta}    
Let $G_1$ and $G_2$ be finite perfect simple graphs on $[d]$.
Then one has
\[
\delta(\Omega(\Qc_{G_1},\Qc_{G_2}), \lambda) 
= \delta(\Gamma(\Qc_{\widehat{G_1}}, \Qc_{\widehat{G_2}}), \lambda)
= (1 + \lambda)\delta(\Gamma(\Qc_{G_1},\Qc_{G_2}), \lambda).
\]
Thus in particular 
\[
{\rm vol}(\Omega(\Qc_{G_1},\Qc_{G_2})) 
= {\rm vol}(\Gamma(\Qc_{\widehat{G_1}}, \Qc_{\widehat{G_2}})) 
= 2 \cdot {\rm vol}(\Gamma(\Qc_{{G_1}},{\Qc_{G_2}})).
\] 
\end{Theorem}

A proof of Theorem \ref{delta} will be given in Section $3$.  
Even though the Ehrhart $\delta$-polynomial of
$\Omega(\Qc_{G_1},\Qc_{G_2})$ coincides with that of  
$\Gamma(\Qc_{\widehat{G_1}},\Qc_{ \widehat{G_2}})$,
$\Omega(\Qc_{G_1},\Qc_{G_2})$ may not be unimodularly equivalent to 
$\Gamma(\Qc_{\widehat{G_1}}, \Qc_{ \widehat{G_2}})$ (Example \ref{e4}). 

\section{Squarefree Gr\"{o}bner basis}
In this section, we prove Theorem \ref{Gorenstein} by using the theory of Gr\"obner bases and toric ideals.
At first, we recall basic materials and notation on toric ideals.
Let $K[{\bf t^{\pm 1}}, s] = 
K[t_{1}^{\pm 1}, \ldots, t_{d}^{\pm 1}, s]$
denote the Laurent polynomial ring in $d + 1$ variables over a field $K$.
For  an integer vector ${\bf a}= [a_{1}, \ldots, a_{d}] ^{\top} \in \ZZ^d$,
the transpose of $[a_{1}, \ldots, a_{d}]$,
${\bf t}^{\ab}s$ is the Laurent monomial
$t_{1}^{a_{1}} \cdots t_{d}^{a_{d}}s \in K[{\bf t^{\pm1}}, s]$. 
Given an integer $d \times n$ matrix 
$A = [{\bf a}_{1}, \ldots, {\bf a}_{n}]$, where
${\bf a}_{j} = [a_{1j}, \ldots, a_{dj}]^{\top}$
 is
the $j$th column of $A$,
then we define 
the {\em toric ring} $K[A]$ of $A$ as follows:
\[
K[A]=K[{\bf t}^{\ab_1}s,\ldots,{\bf t}^{\ab_n}s] \subset K[{\bf t^{\pm 1}}, s].
\]
Let $K[{\bf x}] = K[x_{1}, \ldots, x_{n}]$ be the polynomial ring
in $n$ variables over $K$ and define the surjective ring homomorphism
$\pi : K[{\bf x}] \to K[A]$ by setting $\pi(x_{j}) = {\bf t}^{{\bf a}_{j}}s$
for $j = 1, \ldots, n$.  The {\em toric ideal} of $A$ is the kernel $I_{A}$
of $\pi$.  
Let $<$ be a monomial order on $K[{\bf x}]$ and 
${\rm in}_{<}(I_{A})$ the initial ideal of $I_{A}$ with respect to $<$.
The initial ideal ${\rm in}_{<}(I_{A})$ is called {\em squarefree} if 
${\rm in}_{<}(I_{A})$ is generated by squarefree monomials.
The {\em reverse lexicographic order} on $K[\xb]$ induced by the ordering $x_n <_{\text{rev}}  \cdots <_{\text{rev}} x_1$ is the total order $<_{\text{rev}}$ on the set of monomials in the variables $x_1, x_2,\ldots,x_n$ by setting
$x_1^{a_1}x_2^{a_2}\cdots x_n^{a_n} <_{\text{rev}} x_1^{b_1}x_2^{b_2}\cdots x_n^{b_n}$  if either (i) $\sum_{i=1}^{n}a_i < \sum_{i=1}^{n} b_i$, or (ii) $\sum_{i=1}^{n}a_i = \sum_{i=1}^{n} b_i$ and the rightmost nonzero component of the vector $(b_1-a_1,b_2-a_2,\ldots,b_n-a_n)$ is negative.
A reverse lexicographic order is also called a {\em graded reverse lexicographic order}.
Please refer  \cite[Chapters 1 and 5]{dojoEN} 
for more details on Gr\"obner bases and toric ideals. 

Let $\ZZ_{\geq 0}^{d}$ denote the set of integer column vectors
$[a_{1}, \ldots, a_{d}]^{\top}$ with each $a_{i} \geq 0$,
and let $\ZZ_{\geq 0}^{d \times n}$ denote the set of $d \times n$ 
integer matrices $(a_{ij})_{1 \leq i \leq d \atop 1 \leq j \leq n}$
with each $a_{ij} \geq 0$.
In \cite{harmony}, the concept that $A \in \ZZ_{\geq 0}^{d \times n}$ and $B \in \ZZ_{\geq 0}^{d \times m}$ are of \textit{of harmony} is introduced.
For an integer vector ${\bf a} = [a_{1}, \ldots, a_{d}]^{\top} \in \ZZ^{d}$, 
let ${\bf a}^{(+)} = [a_{1}^{(+)}, \ldots, a_{d}^{(+)}]^{\top}, 
{\bf a}^{(-)} = [a_{1}^{(-)}, \ldots, a_{d}^{(-)}]^{\top} \in \ZZ_{\geq 0}^{d}$ 
where $a_{i}^{(+)} = \max \{0, a_{i} \}$ and  $a_{i}^{(-)} = \max\{0, -a_{i}\}$.
Note that ${\bf a} = {\bf a}^{(+)} - {\bf a}^{(-)}$ holds in general. 
Given $A \in \ZZ_{\geq 0}^{d \times n}$
and $B \in \ZZ_{\geq 0}^{d \times m}$ such that the zero vector
${\bf 0}_{d} = [0, \ldots, 0]^{\top} \in \ZZ^{d}$ is a column in each of $A$ and $B$, we say that $A$ and $B$ 
are {\em of harmony} if the following condition is satisfied:
Let ${\bf a}$ be a column of $A$ 
and ${\bf b}$ that of $B$.  
Let ${\bf c} = {\bf a} - {\bf b} \in \ZZ^{d}$. 
If ${\bf c} = {\bf c}^{(+)} - {\bf c}^{(-)}$,
then ${\bf c}^{(+)}$ is a column vector of $A$
and ${\bf c}^{(-)}$ is a column vector of $B$.

Now we prove the following theorem.
\begin{Theorem}
	\label{squarefree}
	Let $A = [{\bf a}_{1}, \ldots, {\bf a}_{n}] \in \ZZ_{\geq 0}^{d \times n}$ 
	and 
	$B = [{\bf b}_{1}, \ldots, {\bf b}_{m}]\in \ZZ_{\geq 0}^{d \times m}$, where $\ab_{n}=\bb_{m}={\bf 0}_d \in \ZZ^d$, 
	be of harmony.
	Let $K[ {\bf x}] = K[x_{1}, \ldots, x_{n}]$
	and $K[ {\bf y}] = K[ y_{1}, \ldots, y_{m}]$
	be the polynomial rings over a field $K$. 
	Suppose that 
	${\rm in}_{<_A}(I_{A})
	\subset K[{\bf x}]$ and 
	${\rm in}_{<_B}(I_{B}) \subset K[ {\bf y}]$ are
	squarefree with respect to reverse lexicographic orders 
	$<_A$ on $K[{\bf x}]$
	and $<_B$ on $K[{\bf y}]$ respectively
	satisfying the condition that
	\begin{itemize}
		\item
		$x_i  <_A  x_j$ \,if\, 
		for each $1 \leq k \leq d$ 
		$a_{ki} \leq a_{kj}$. 
		\item $x_{n}$ is the smallest variable with respect to $<_A$.
		\item $y_{m}$ is the smallest variable with respect to $<_B$.
	\end{itemize}
	Let $[-B, A]^{*}$ denote the $(d+1) \times (n + m+1)$
	integer matrix 
	$$\begin{bmatrix}
	-\bb_1 & \cdots & -\bb_m  & \ab_1 & \cdots & \ab_n & \mathbf{0}_d \\
	-1 & \cdots & -1 & 1 & \cdots & 1& 0
	\end{bmatrix}.$$
	Then the toric ideal $I_{[-B, A]^{*}}$ of
	$[-B, A]^{*}$ possesses a squarefree initial ideal 
	with respect to a reverse lexicographic order
	whose smallest variable corresponds to the column
	${\bf 0}_{d+1} \in \ZZ^{d+1}$ of $[-B, A]^{*}$.
\end{Theorem}
\begin{proof}
	Let $I_{[-B, A]^{*}} \subset K[\xb, \yb, z]
	= K[x_1, \ldots, x_n, y_1, \ldots, y_m, z]$ 
	be the toric ideal of $[-B, A]^{*}$
	defined by the kernel of 
	$$\pi^{*} : K[\xb, \yb, z] \to K[[-B, A]^{*}]\subset K[ t_{1}^{\pm 1}, \ldots, t_{d+1}^{\pm 1}, s]$$
	with
	$\pi^{*}(z) = s$,
	$\pi^{*}(x_i) = {\bf t}^{{\bf a}_{i}}t_{d+1}s$
	for $i = 1, \ldots, n$ and
	$\pi^{*}(y_j) = {\bf t}^{ - {\bf b}_{j}}t_{d+1}^{-1}s$
	for $j = 1, \ldots, m$.
	Assume that the reverse lexicographic orders
	$<_A$ and $<_B$ are induced by the orderings
	$x_n <_A \cdots <_A x_1$
	and
	$y_m <_B \cdots <_B y_1$.
	Let $<_{\rm rev}$ be the reverse lexicographic order on 
	$K[\xb, \yb, z]$ induced by the ordering
	$$ z <_{\rm rev} x_n <_{\rm rev} \cdots <_{\rm rev} x_1 <_{\rm rev} y_m <_{\rm rev} \cdots <_{\rm rev} y_1.$$
	
	In general, for an integer vector $\ab = [a_{1}, \ldots, a_{d}]^{\top} \in \ZZ^{d}$, 
	we let ${\rm supp}(\ab)=\{i :  1\leq i \leq d, a_i \neq 0\}$.
	Set the following:
$$			{\mathcal E}
			=
			\{ \, (i,j) \, : \, 1 \leq i \leq n, \, 1 \leq j \leq m, \, 
			{\rm supp} (\ab_i) \cap {\rm supp} (\bb_j) \neq 
			\emptyset \, \}. 
	$$
	If ${\bf c} = \ab_i - \bb_j$
	with $(i,j) \in {\mathcal E}$,
	then it follows that ${\bf c}^{(+)} \neq \ab_i$
	and
	${\bf c}^{(-)} \neq \bb_j$.
	Since $A$ and $B$ are of harmony,  we know that
	${\bf c}^{(+)}$ is a column of $A$
	and ${\bf c}^{(-)}$ is a column of $B$.
	It follows that $f = x_i y_j -x_ky_{\ell}$ ($\neq 0$)
	belongs to $I_{[-B, A]^{*}}$, where
 ${\bf c}^{(+)} = \ab_k$
 and
	${\bf c}^{(-)} = \bb_\ell$. 
	Then since 
	for each $1 \leq  c \leq d$, 
	$a_{ck} \leq a_{ci}$, 
	one has $x_k <_A x_i$ and ${\rm in}_{<_{\rm rev}}(f) = x_iy_j$. 
	Hence
	\[
	\{ \, x_i y_j \, : \, (i,j) \in {\mathcal E} \, \}
	\subset
	{\rm in}_{<_{\rm rev}}(I_{[-B, A]^{*}}).
	\]
	Moreover, it follows that $x_ny_m-z^2 \in I_{[-B, A]^{*}}$ and $x_ny_m \in {\rm in}_{<_{\rm rev}}
	( I_{[-B, A]^{*}})$. 
	We set
	\[
	{\mathcal M}=
	\{x_{n}y_{m}\}
	\cup	
	\{ \, x_i y_j \, : \, (i,j) \in {\mathcal E} \, \}
	\cup 
	{\mathcal M}_A
	\cup 
	{\mathcal M}_B \, \, \,  
	( \, \subset 
	{\rm in}_{<_{\rm rev}}( I_{[-B, A]^{*}}) \, ),
	\]
	where
	${\mathcal M}_{A}$ (resp. ${\mathcal M}_B$) is
	the minimal set of squarefree monomial generators of 
	${\rm in}_{<_A} (I_{A})$
	(resp. ${\rm in}_{<_B} (I_{B})$).
	Let $\Gc$ be a finite set of binomials belonging to $I_{[-B, A]^{*}}$ with ${\mathcal M} = \{ {\rm in}_{<_{\rm rev}}(f) : f \in {\mathcal G}\}$.
	
	Now, we prove that $\Gc$ is a Gr\"obner base of ${\rm in}_{<_{\rm rev}}( I_{[-B, A]^{*}})$
	with respect to $<_{\rm rev}$.
	By the following fact (\cite[(0.1), p.~1914]{OHrootsystem}) on Gr\"obner bases,
	we must prove the following assertion:
	If $u$ and $v$ are monomials belonging to $K[\xb,\yb,z]$ with $u \neq v$ such that
	$u \notin ( {\rm in}_<(g) : g \in {\mathcal G})$ and $v\notin( {\rm in}_<(g) : g \in {\mathcal G})$ 
	, then $\pi^*(u) \neq \pi^*(v)$.
	
	Suppose that there exists a nonzero irreducible binomial $g = u - v$ be  belonging to $I_{[-B, A]^{*}}$ such that 
	$u \notin ( {\rm in}_<(g) : g \in {\mathcal G})$ and $v\notin( {\rm in}_<(g) : g \in {\mathcal G})$.
	Write
	\[
	u = \left(\,
	\prod_{p \in P} x_{p}^{i_p} 
	\right)
	\left(\, 
	\prod_{q \in Q} y_q^{j_q} 
	\right),
	\, \, \, \, \, \, 
	v =
	z^{\alpha} 
	\left(\,
	\prod_{p' \in P'} x_{p'}^{i'_{p'}}
	\right) 
	\left(\, 
	\prod_{q' \in Q} y_{q'}^{j'_{q'}} 
	\right),
	\]
	where $P$ and $P'$ are subsets of $[n]$, 
    $Q$ and $Q'$ are subsets of $[m]$, 
	$\alpha$ is a nonnegative integer, 
	and  each of $i_p, j_q, i'_{p'}, j'_{q'}$ is a positive integer.
	Since $g = u -v$ is irreducible, one has
	$P\cap P'  = Q \cap Q' = \emptyset$.
	Furthermore, by the fact that each of $x_i y_j$ with $(i,j) \in {\mathcal E}$ 
	can divide neither $u$ nor $v$, it follows that
	$$
	\left(\,
	\bigcup_{p \in P} {\rm supp} (\ab_p) 
	\right)
	\cap 
	\left(\,
	\bigcup_{q \in Q} {\rm supp} (\bb_q) 
	\right)
	=
	\left(\,
	\bigcup_{p' \in P'} {\rm supp} (\ab_{p'}) 
	\right)
	\cap 
	\left(\,
	\bigcup_{q' \in Q'} {\rm supp} (\bb_{q'}) 
	\right)
	=
	\emptyset.
	$$
	Hence, since $\pi^{*} (u) = \pi^{*} (v)$,
	 it follows that
	$$
	\sum_{p \in P} i_p \ab_p = \sum_{p' \in P'} i'_{p'} \ab_{p'},
	\, \, \, \, \,  
	\sum_{q \in Q} j_q \bb_q = \sum_{q' \in Q'} j'_{q'} \bb_{q'}.
	$$
	Let $\xi = \sum_{p \in P} i_p$,
	$\xi' = \sum_{p' \in P'} i'_{p'}$,
	$\nu = \sum_{q \in Q} j_q$,
	and $\nu' = \sum_{q' \in Q'} j'_{q'}$. 
	Then
	$\xi  + \nu =  \xi' + \nu'+\alpha$.
	Since $\alpha \geq 0$, it follows that
	either $\xi \geq \xi'$ or $\nu \geq \nu'$.
	Assume that $\xi > \xi'$.
	Then
	\[
	h = \prod_{p \in P} x_{p}^{i_p}
	-
	x_{n}^{\xi - \xi'} \left(\, \prod_{p' \in P'} x_{p'}^{i'_{p'}} \right)
	\]
	belongs to  $I_{A}$ and $I_{[-B, A]^{*}}$.
	If $h \neq 0$, then ${\rm in}_{<_{A}}(h) ={\rm in}_{<_{\rm rev}}(h) = \prod_{p \in P} x_{p}^{i_p}$
	divides $u$, a contradiction.  
	Hence $P=\{n\}$ and $Q=\emptyset$.
	If $\xi = \xi'$, then the binomial
		$$h_{0} = \prod_{p \in P} x_{p}^{i_p} 
		-
		\prod_{p' \in P'} x_{p'}^{i'_{p'}}$$
	belongs to $I_{A}$ and $I_{[-B, A]^{*}}$.  
	Moreover, if $h_{0} \neq 0$, then 
	either
	$\prod_{p \in P} x_{p}^{i_p}$ 
	or
	$\prod_{p' \in P'} x_{p'}^{i'_{p'}}$
	must belong to 	${\rm in}_{<_A} (I_{A})$ and ${\rm in}_{<_{\rm rev}}(I_{[-B, A]^{*}})$.
	This contradicts the fact that
	each of $u$ and $v$
	can be divided by none of the monomials belonging to 
	${\mathcal M}$.  Hence $h_{0} = 0$
	and $P = P' = \emptyset$.
	Similarly, $Q=\{m\}$ and $Q'=\emptyset$, or $Q = Q' = \emptyset$. 
	Hence we know that 
	$g=x_{n}^{k}y_{m}^{\ell}-z^{\alpha}$, where $k$ and $\ell$ are nonnegative integers.
	Since $u$ cannot be divided by $x_{n}y_{m}$,
	it follows that $g=0$, a contradiction.
Therefore, $\Gc$ is a Gr\"obner base of ${\rm in}_{<_{\rm rev}}( I_{[-B, A]^{*}})$
with respect to $<_{\rm rev}$.
\end{proof}

Now, we recall the following lemma.
\begin{Lemma}[{\cite[Lemma 1.1]{HMOS}}]
	\label{HMOS}
	Let $\Pc \subset \RR^d$ be a lattice polytope of dimension $d$  such that the origin of $\RR^d$ is contained 
	in its interior and $\Pc \cap \ZZ^d=\{\ab_1,\ldots,\ab_n \}$.
	Suppose that any integer point in $\ZZ^{d+1}$ is a linear integer combination of the integer points in $\Pc \times \{1\}$ and there exists an ordering of the variables $x_{i_1} < \cdots < x_{i_n}$ for which $\ab_{i_1}= \mathbf{0}_d$ such that the initial ideal $\textnormal{in}_{<}(I_{A})$ of the toric ideal $I_{A}$ with respect to the reverse lexicographic order $<$ on the polynomial ring $K[x_1,\ldots,x_n]$
	induced by the ordering is squarefree,
	where $A=[\ab_1,\ldots,\ab_n]$.
	Then $\Pc$ is a reflexive polytope which possesses the integer decomposition property.
\end{Lemma}

By Theorem \ref{squarefree} and this lemma, we obtain the following corollary.
\begin{Corollary}
	\label{GorCor}
	Work with the same situation as in Theorem \ref{squarefree}.
		Let $\Pc \subset \RR^{d+1}$ be the lattice polytope of dimension $d+1$
		with $$\Pc \cap \ZZ^{d+1}=\left\{ 
		\left[ \begin{array}{c}
		{\bf a}_{1}\\
		1
		\end{array}
		\right]
		, \ldots,
		\left[ \begin{array}{c}
		{\bf a}_{n}\\
		1
		\end{array}
		\right]
		,
		\left[ \begin{array}{c}
		-{\bf b}_{1}\\
		-1
		\end{array}
		\right], \ldots, 
		\left[ \begin{array}{c}
		-{\bf b}_{m}\\
		-1
		\end{array}
		\right],
		   \mathbf{0}_{d+1}\right\}.$$
		Suppose that $\mathbf{0}_{d+1} \in \ZZ^{d+1}$ belongs to the interior of $\Pc$ and 
		any integer point in $\ZZ^{d+2}$ is a linear integer combination of the integer points in $\Pc \times \{1\}$.
		Then $\Pc$ is a reflexive polytope which possesses the integer decomposition property. 
\end{Corollary}

Recall that an integer matrix $A$
is {\em compressed} (\cite{compressed}, \cite{Sul}) if the initial ideal 
of the toric ideal $I_{A}$ is squarefree with respect to 
any reverse lexicographic order.  

Finally, we prove Theorem \ref{Gorenstein}.
\begin{proof}[Proof of Theorem \ref{Gorenstein}]
For a finite simple graph $G$ on $[d]$,
let $A_{S(G)}$ be the matrix whose columns are those $\rho (W)$ with $W \in S(G)$.
If $W \in S(G)$, then  each subset of $W$ is also a stable set of $G$.
This means that $S(G)$ is a simplicial complex on $[d]$.
Hence it is easy to show that $A_{S(G_1)}$ and $A_{S(G_2)}$ are of harmony.
Moreover, for any perfect graph $G$, $A_{S(G)}$ is compressed (\cite[Example 1.3 (c)]{compressed}).
Let $\Pc \subset \RR^{d+1}$ be the convex hull of $\{\pm(\eb_1+\eb_{d+1}),\ldots,\pm(\eb_d+\eb_{d+1}),\pm\eb_{d+1}\}$.
Then it follows that  $\mathbf{0}_{d+1} \in \ZZ^{d+1}$ belongs to the interior of $\Pc$ and 
any integer point in $\ZZ^{d+2}$ is a linear integer combination of the integer points in $\Pc \times \{1\}$.
Moreover, we have $\Pc \subset \Omega(\Qc_{G_1}, \Qc_{G_2})$.
This implies that $\mathbf{0}_{d+1} \in \ZZ^{d+1}$ belongs to the interior of $\Omega(\Qc_{G_1}, \Qc_{G_2})$
and any integer point in $\ZZ^{d+2}$ is a linear integer combination of the integer points in $\Omega(\Qc_{G_1}, \Qc_{G_2}) \times \{1\}$.
On the other hand,
one has 
$$\Omega(\Qc_{G_1}, \Qc_{G_2}) \cap \{[a_1,\ldots,a_{d+1}]^{\top} \in \RR^{d+1} : a_{d+1}=0 \}=\dfrac{1}{2}(\Qc_{G_1} - \Qc_{G_2}) \times \{0\}.$$
Since $\dfrac{1}{2}(\Qc_{G_1} - \Qc_{G_2}) \cap \ZZ^d=\{{\bf 0}_d \}$,  we obtain
$$\Omega(\Qc_{G_1}, \Qc_{G_2}) \cap \ZZ^{d+1} =\left\{	\left[ \begin{array}{c}
{\bf a}\\
1
\end{array}
\right] : \ab \in \Qc_{G_1} \cap \ZZ^d \right\} \cup \left\{	\left[ \begin{array}{c}
-{\bf b}\\
-1
\end{array}
\right] : \bb \in \Qc_{G_2} \cap \ZZ^d \right\} \cup \{{\bf 0}_{d+1}\}.$$
Hence, by Corollary \ref{GorCor}, 
if $G_1$ and $G_2$ are perfect, $\Omega(\Qc_{G_1}, \Qc_{G_2})$ is a reflexive polytope which possesses the integer decomposition property.

Next, we prove that if $G_1$ is not perfect, then $\Omega(\Qc_{G_1}, \Qc_{G_2})$ does not possess the integer decomposition property.
Assume that $G_1$ is not perfect and $\Omega(\Qc_{G_1}, \Qc_{G_2})$ possesses the integer decomposition property.
By the strong perfect graph theorem (\cite{sptheorem}),  $G_1$ possesses either an odd hole or an odd antihole,
where an odd hole is an induced odd cycle of length $\geq 5$ and an odd antihole is the complementary graph of an odd hole.
Suppose that $G_1$ possesses an odd hole $C$ of length $2\ell+1$, where $\ell \geq 2$ and we regard $C$ as a finite graph on $[d]$.
We may assume that the edge set of $C$ is $\{\{i,i+1\} : 1\leq i \leq 2\ell \} \cup \{1,2\ell+1\} $.
Then the maximal stable sets of $C$ in $[2\ell+1]$ are 
$$S_1=\{1,3,\ldots,2\ell-1\},S_2=\{2,4,\ldots,2\ell \},\ldots,S_{2\ell+1}=\{2\ell+1,2,4,\ldots,2\ell-2\}$$
and each $i \in [2\ell+1]$ appears $\ell$ times in the above list.
For $1\leq i \leq 2\ell+1$, we set
$\vb_{i}=\sum_{j \in S_i}\eb_j+\eb_{d+1}$.
Then one has
$$\ab=\dfrac{\vb_1+\cdots+\vb_{2\ell+1}+(-\eb_{d+1})}{\ell}=\eb_1+\cdots+\eb_{2\ell+1}+2\eb_{d+1}.$$
Since $2 <(2\ell+2)/\ell \leq 3$, $\ab \in 3\Omega(\Qc_{G_1}, \Qc_{G_2})$.
Hence there exist $\ab_1,\ab_2,\ab_3 \in \Omega(\Qc_{G_1}, \Qc_{G_2}) \cap \ZZ^{d+1}$ such that $\ab=\ab_1+\ab_2+\ab_3$.
Then we may assume that $\ab_1,\ab_2 \in \Qc_{C}\times \{1\}$ and $\ab_3={\bf 0}_{d+1}$. 
However, since the maximal cardinality of the stable sets of $C$ in $[2\ell+1]$ equals $\ell$, a contradiction.

Suppose that $G_1$ possesses an odd antihole $C$ such that the length of $\overline{C}$ equals $2\ell+1$, where $\ell \geq 2$ and we regard $\overline{C}$ as a finite graph on $[d]$.
Similarly, we may assume that the edge set of $\overline{C}$ is $\{\{i,i+1\} : 1\leq i \leq 2\ell \} \cup \{1,2\ell+1\} $.
Then the maximal stable sets of $C$ are the edges of $\overline{C}$.
For $1 \leq i \leq 2\ell$, we set $\wb_i=\eb_i+\eb_{i+1}+\eb_{d+1}$ and set $\wb_{2\ell+1}=\eb_1+\eb_{2\ell+1}+\eb_{d+1}$.
Then one has
$$\bb=\dfrac{\wb_1+\cdots+\wb_{2\ell+1}+(-\eb_{d+1})}{2}=\eb_1+\cdots+\eb_{2\ell+1}+\ell\eb_{d+1}$$
and $\bb \in (\ell+1)\Omega(\Qc_{G_1}, \Qc_{G_2})$.
Hence there exist $\bb_1,\ldots,\bb_{\ell+1} \in \Omega(\Qc_{G_1}, \Qc_{G_2}) \cap \ZZ^{d+1}$ such that $\bb=\bb_1+\cdots+\bb_{\ell+1}$.
Then we may assume that $\bb_1,\ldots,\bb_{\ell} \in \Qc_{C}\times \{1\}$ and $\bb_{\ell+1}={\bf 0}_{d+1}$. 
However, since the maximal cardinality of the stable sets of $C$ in $[2\ell+1]$ equals $2$, a contradiction.

Therefore, if $\Omega(\Qc_{G_1}, \Qc_{G_2})$ possesses the integer decomposition property,
then $G_1$ and $G_2$ are perfect, as desired.
\end{proof}

\section{Ehrhart $\delta$-polynomials}
In this section, 
we consider the Ehrhart $\delta$-polynomials and the volumes of the polytopes $\Omega(\Qc_{G_1},\Qc_{G_2})$  and $\Gamma(\Qc_{G_1},\Qc_{G_2})$,
in particular, we prove Theorem \ref{delta}.
Let $\Pc \subset \RR^d$ be a lattice polytope of dimension $d$ with $\Pc \cap \ZZ^d=\{\ab_1,\ldots,\ab_n\}$.
Set $A=[\ab_1,\ldots,\ab_n]$.
We define the toric ring $K[\Pc]$ and the toric ideal $I_{\Pc}$ of $\Pc$ by $K[A]$ and  $I_{A}$.
In order to prove Theorem \ref{delta}, we use the following facts. 
\begin{itemize}
	\item  
	If $\Pc$ possesses the integer decomposition property, then the \textit{Ehrhart polynomial} $\mid n\Pc \cap \ZZ^d \mid $ of $\Pc$ is equal to the Hilbert polynomial of 
	the toric ring $K[\Pc]$.  
	\item Let $S$ be a polynomial ring and $I \subset S$ be a graded ideal of $S$. 
	Let $<$ be a monomial order on $S$. 
	Then $S/I$ and $S/\mathrm{in}_{<}(I)$ have the same Hilbert function. 
	(see \cite[Corollary 6.1.5]{Monomial}).
\end{itemize}

Now, we prove the following theorem.
\begin{Theorem}
	\label{ehrhart}
	Work with the same situation as in Theorem \ref{squarefree}.
	Let $\Pc \subset \RR^d$ be the lattice polytope with $\Pc \cap \ZZ^d=\{\ab_1,\ldots,\ab_{n}\}$ and $\Qc \subset \RR^{d}$ the lattice polytope with $\Qc \cap \ZZ^d=\{\bb_1,\ldots,\bb_{m}\}$.
	Suppose that any integer point in $\ZZ^{d+1}$ is a linear integer combination of the integer points in $\Gamma(\Pc,\Qc) \times \{1\}$,
	any integer point in $\ZZ^{d+2}$ is a linear integer combination of the integer points in $\Omega(\Pc,\Qc) \times \{1\}$,	
$$\Gamma(\Pc,\Qc) \cap \ZZ^{d}=\{{\bf a}_{1}, \ldots, {\bf a}_{n-1} ,-\bb_1,\ldots,-\bb_{m-1},\mathbf{0}_{d}\}$$
and	
	$$\Omega(\Pc,\Qc) \cap \ZZ^{d+1}=\left\{ 
	\left[ \begin{array}{c}
	{\bf a}_{1}\\
	1
	\end{array}
	\right]
	, \ldots,
	\left[ \begin{array}{c}
	{\bf a}_{n}\\
	1
	\end{array}
	\right]
	,
	\left[ \begin{array}{c}
	-{\bf b}_{1}\\
	-1
	\end{array}
	\right], \ldots, 
	\left[ \begin{array}{c}
	-{\bf b}_{m}\\
	-1
	\end{array}
	\right],
	\mathbf{0}_{d+1}\right\}.$$
	Then we obtain 
	$$\delta(\Omega(\Pc,\Qc),\lambda)=(1+\lambda)\delta(\Gamma(\Pc,\Qc),\lambda).$$
	In particular, 
		$$\textnormal{vol}(\Omega(\Pc,\Qc))=2 \cdot \textnormal{vol}(\Gamma(\Pc,\Qc)).$$
\end{Theorem}
\begin{proof}
	Set $\Rc=\textnormal{conv}\{\Gamma(\Pc,\Qc) \times \{0\}, \pm \eb_{d+1}\}$.
	Then it follows from \cite[Theorem 1.4]{BeckJayawantMcAllister} that $\delta(\Rc,\lambda)=(1+\lambda)\delta(\Gamma(\Pc,\Qc),\lambda).$
	Moreover, by \cite[Theorem 1.1]{harmony} and Theorem \ref{squarefree},
	$\Rc$ and $\Omega(\Pc,\Qc)$ possess the integer decomposition property.
	Hence we should show that $K[\Rc]$ and $K[\Omega(\Pc,\Qc)]$ have the same Hilbert function.
	
	Now, use the same notation as in the proof of Theorem \ref{squarefree}.
	Then we have $$\dfrac{K[\xb,\yb,z]}{\text{in}_{<_{\text{ref}}}(I_{\Omega(\Pc,\Qc)})}=\dfrac{K[\xb,\yb,z]}{(\Mc)}.$$
	Set
	\begin{displaymath}
		\ab_i'=
		\begin{cases}
				\left[ \begin{array}{c}
			{\bf a}_{i}\\
			0
			\end{array}
			\right], &1\leq i \leq n-1,\\
			\eb_{d+1}, & i=n,\\
			{\bf 0_{d+1}}, & i=n+1,
		\end{cases}
		\ \text{and} \ 	\bb_j'=
		\begin{cases}
				\left[ \begin{array}{c}
			{\bf b}_{i}\\
			0
			\end{array}
			\right], &1\leq j \leq m-1,\\
			\eb_{d+1}, & j=m,\\
			{\bf 0_{d+1}}, &j=m+1.
		\end{cases}
	\end{displaymath}
	Then it is easy to show that $A'=[\ab_1',\ldots,\ab_{n+1}']$ and $B'=[\bb_1',\ldots,\bb_{m+1}']$ are of harmony.
	Moreover, $\text{in}_{<_{B'}}(I_{B'}) \subset K[y_1,\ldots,y_{m+1}]$ and $\text{in}_{<_{A'}}(I_{A'}) \subset K[x_1,\ldots,x_{n+1}]$ are squarefree with respect to reverse lexicographic orders $<_{A'}$ on $K[x_1,\ldots,x_{n+1}]$ and $<_{B'}$ on $K[y_1,\ldots,y_{m+1}]$ induced by the orderings
		 $ x_{n+1} <_{A'} x_n <_{A'} \cdots <_{A'} x_1$
		 and  
		 $y_{m+1} <_{B'} y_m <_{B'} \cdots <_{B'} y_1$.
	Now, we introduce the following:
	\begin{displaymath}
		\begin{aligned}
			{\mathcal E}'
			=
			\{ \, (i,j) \, : \, 1 \leq i \leq n, \, 1 \leq j \leq m, \, 
			{\rm supp} (\ab_i') \cap {\rm supp} (\bb_j') \neq 
			\emptyset \, \}. 
		\end{aligned}
	\end{displaymath}
	Then we have $\Ec'=\Ec \cup \{(n,m)\}
	$.
	Let 
	${\mathcal M}_{A'}$ (resp. ${\mathcal M}_{B'}$) be
	the minimal set of squarefree monomial generators of 
	${\rm in}_{<_{A'}} (I_{A'})$
	(resp. ${\rm in}_{<_{B'}} (I_{B'})$).
	Then it follows that $\Mc_{A'}=\Mc_A$ and $\Mc_{B'}=\Mc_{B}$.
	This says that $\Mc=\Ec' \cup \Mc_{A'} \cup \Mc_{B'}$.
	By the proof of \cite[Theorem 1.1]{harmony}, we obtain
	$\text{in}_{<_{\text{rev}}}(I_{\Rc})=(\Mc) \subset K[\xb,\yb,z]$.
	Hence it follows that 
	$$\dfrac{K[\xb,\yb,z]}{\text{in}_{<_{\text{rev}}}(I_{\Omega(\Pc,\Qc)})}=\dfrac{K[\xb,\yb,z]}{\text{in}_{<_{\text{rev}}}(I_{\Rc})}.$$
	Therefore, $K[\Rc]$ and $K[\Omega(\Pc,\Qc)]$ have the same Hilbert function, as desired.
\end{proof}


Now, we prove Theorem \ref{delta}.
\begin{proof}[Proof of Theorem \ref{delta}]
	For any finite simple graph $G$ on $[d]$,
	we have $S(\widehat{G})=S(G) \cup \{d+1\}$.
	Hence it follows that
	$\Gamma(\Qc_{\widehat{G_1}},\Qc_{\widehat{G_2}})=\textnormal{conv}\{\Gamma(\Qc_{G_1},\Qc_{G_1})\times \{0\}, \pm \eb_{d+1}\}$. 
	Therefore, by Theorem \ref{ehrhart}, we obtain 
	$$\delta(\Omega(\Qc_{G_1},\Qc_{G_2},\lambda)=\delta(\Gamma(\Qc_{\widehat{G_1}},\Qc_{\widehat{G_2}}),\lambda)=(1+\lambda)\delta(\Gamma(\Qc_{G_1},\Qc_{G_2}),\lambda),$$
	as desired.
\end{proof}
\section{Examples}
In this section, we give some curious examples of $\Gamma(\Qc_{G_1},\Qc_{G_2})$ and $\Omega(\Qc_{G_1},\Qc_{G_2})$.
At first, the following example says that even though $G_1$ and $G_2$ are not perfect, $\Omega(\Qc_{G_1},\Qc_{G_2})$ may be reflexive.
\begin{Example}
	\label{e1}
	Let $G$ be the finite simple graph as follows:
	\newline
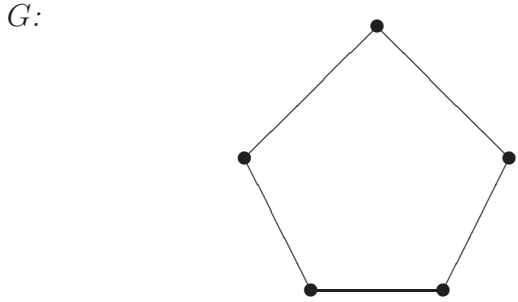
	\begin{picture}(400,150)(10,50)
	\put(60,170){$G$:}
	\put(200,170){\circle*{5}}
	\put(250,120){\circle*{5}}
	\put(150,120){\circle*{5}}
	\put(175,70){\circle*{5}}
	\put(225,70){\circle*{5}}
	\put(225,70){\line(-1,0){50}}
	\put(175,70){\line(-1,2){25}}
	\put(150,120){\line(1,1){50}}
	\put(200,170){\line(1,-1){50}}
	\put(250,120){\line(-1,-2){25}}
	\end{picture}\\
	Namely, $G$ is a cycle of length $5$.
	Then $G$ is not perfect.
	Hence $\Gamma(\Qc_{G},\Qc_{G})$ is not reflexive.
	However, $\Omega(\Qc_{G},\Qc_{G})$ is reflexive.
	In fact, we have
	$$\delta(\Gamma(\Qc_{G},\Qc_{G}),\lambda)=1+15\lambda+60\lambda^2+62\lambda^3+15\lambda^4+\lambda^5,$$
	$$\delta(\Omega(\Qc_{G},\Qc_{G}),\lambda)=1+16\lambda+75\lambda^2+124\lambda^3+75\lambda^4+16\lambda^5+\lambda^6.$$
	Moreover, $\Gamma(\Qc_{G},\Qc_{G})$ possesses the integer decomposition property,
	but $\Omega(\Qc_{G},\Qc_{G})$ does not possess the integer decomposition property.
\end{Example}

For this example, $\Gamma(\Qc_{G},\Qc_{G})$ possesses the integer decomposition property.
Next example says that if $G_1$ and $G_2$ are not perfect, 
$\Gamma(\Qc_{G_1},\Qc_{G_2})$ may not possess the integer decomposition property.
\begin{Example}
		\label{e2}
	Let $G$ be a finite simple graph whose complementary graph $\overline{G}$ is as follows:
	\newline
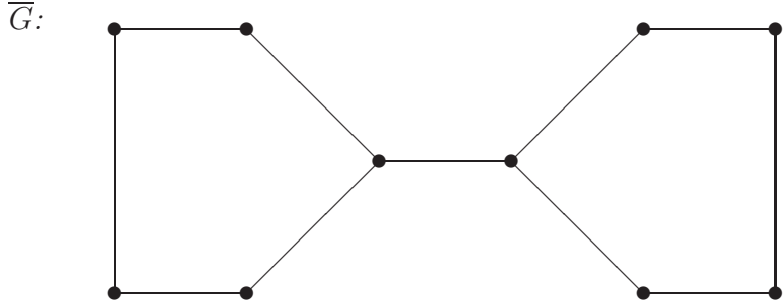
	\begin{picture}(400,150)(10,50)
	\put(60,170){$\overline{G}$:}
	\put(100,170){\circle*{5}}
	\put(150,170){\circle*{5}}
	\put(200,120){\circle*{5}}
	\put(100,70){\circle*{5}}
	\put(150,70){\circle*{5}}
	\put(100,170){\line(1,0){50}}
	\put(150,170){\line(1,-1){50}}
	\put(200,120){\line(-1,-1){50}}
	\put(150,70){\line(-1,0){50}}
	\put(100,70){\line(0,1){100}}
	\put(200,120){\line(1,0){50}}
	\put(300,170){\circle*{5}}
	\put(350,170){\circle*{5}}
	\put(250,120){\circle*{5}}
	\put(300,70){\circle*{5}}
	\put(350,70){\circle*{5}}
	\put(300,170){\line(1,0){50}}
	\put(350,170){\line(0,-1){100}}
	\put(350,70){\line(-1,0){50}}
	\put(300,70){\line(-1,1){50}}
	\put(250,120){\line(1,1){50}}
	\end{picture}\\
	Then $G$ is not perfect.
	Hence $\Gamma(\Qc_{G},\Qc_{G})$ is not reflexive.
	However, $\Omega(\Qc_{G},\Qc_{G})$ is reflexive.
	Moreover, in this case, $\Gamma(\Qc_{G},\Qc_{G})$ and 
	$\Omega(\Qc_{G},\Qc_{G})$ do not possess the integer decomposition property.
\end{Example}

For any finite simple graph $G$ with at most $6$ vertices,
$\Omega(\Qc_{G},\Qc_{G})$ is always reflexive.
However, in the case of finite simple graphs with more than $6$ vertices, we obtain a different result.

\begin{Example}
	\label{e3}
	Let $G$ be the finite simple graph as follows:
	\newline
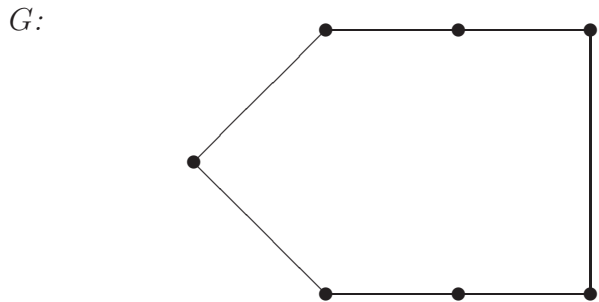
	\begin{picture}(400,150)(10,50)
	\put(60,170){$G$:}
	\put(180,170){\circle*{5}}
	\put(230,170){\circle*{5}}
	\put(280,170){\circle*{5}}
	\put(280,70){\circle*{5}}
	\put(130,120){\circle*{5}}
	\put(180,70){\circle*{5}}
	\put(230,70){\circle*{5}}
	\put(180,170){\line(1,0){50}}
	\put(230,170){\line(1,0){50}}
		\put(280,70){\line(0,1){100}}
	\put(280,70){\line(-1,0){50}}
	\put(230,70){\line(-1,0){50}}
	\put(180,70){\line(-1,1){50}}
	\put(130,120){\line(1,1){50}}
	\end{picture}\\
	Namely, $G$ is a cycle of length $7$. Then $G$ is not perfect.
	Hence $\Gamma(\Qc_{G},\Qc_{G})$ is not reflexive.
	Moreover, $\Omega(\Qc_{G},\Qc_{G})$ is not reflexive.
	In fact, we have
	$$\delta(\Gamma(\Qc_{G},\Qc_{G}),\lambda)=1+49\lambda+567\lambda^2+1801\lambda^3+1799\lambda^4+569\lambda^5+49\lambda^6+\lambda^7,$$
	$$\delta(\Omega(\Qc_{G},\Qc_{G}),\lambda)=1+50\lambda+616\lambda^2+2370\lambda^3+3598\lambda^4+2368\lambda^5+618\lambda^6+50\lambda^7+\lambda^8.$$
\end{Example}

Finally, we show that even though the Ehrhart $\delta$-polynomial of
$\Omega(\Qc_{G_1},\Qc_{G_2})$ coincides with that of  
$\Gamma(\Qc_{\widehat{G_1}},\Qc_{ \widehat{G_2}})$,
$\Omega(\Qc_{G_1},\Qc_{G_2})$ may not be unimodularly equivalent to 
$\Gamma(\Qc_{\widehat{G_1}}, \Qc_{ \widehat{G_2}})$.
\begin{Example}
	\label{e4}
	Let $G$ be the finite simple graph as follows:
	\newline
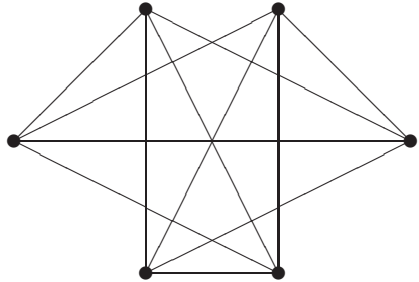
	\begin{picture}(400,150)(10,50)
	\put(60,170){$G$:}
	\put(180,170){\circle*{5}}
	\put(230,170){\circle*{5}}
	\put(280,120){\circle*{5}}
	\put(130,120){\circle*{5}}
	\put(180,70){\circle*{5}}
	\put(230,70){\circle*{5}}
	\put(180,170){\line(0,-1){100}}
	\put(180,170){\line(2,-1){100}}
	\put(180,170){\line(1,-2){50}}
	\put(180,170){\line(-1,-1){50}}
	\put(230,170){\line(0,-1){100}}
	\put(230,170){\line(1,-1){50}}
	\put(230,170){\line(-1,-2){50}}
	\put(230,170){\line(-2,-1){100}}
	\put(280,120){\line(-2,-1){100}}
	\put(280,120){\line(-1,0){150}}
	\put(230,70){\line(-1,0){50}}
	\put(230,70){\line(-2,1){100}}
	\end{picture}\\
	Namely, $G$ is a $(2,2,2)$-complete multipartite graph. Then $G$ is perfect.
	Hence we know that
	$\Omega(\Qc_{G},\Qc_{G})$ and $\Gamma(\Qc_{\widehat{G}},\Qc_{\widehat{G}})$ have the same Ehrhart $\delta$-polynomial and the same volume.
	However, $\Omega(\Qc_{G},\Qc_{G})$ has $54$ facets and $\Gamma(\Qc_{\widehat{G}},\Qc_{\widehat{G}})$ has $432$ facets.
	Hence, $\Omega(\Qc_{G},\Qc_{G})$ and $\Gamma(\Qc_{\widehat{G}},\Qc_{\widehat{G}})$  are not unimodularly equivalent.
	Moreover, for any finite simple graph $G'$ on $\{1,\ldots,7\}$ except for $\widehat{G}$, the Ehrhart $\delta$-polynomial of $\Gamma(\Qc_{G'},\Qc_{G'})$
	is not equal to that of $\Omega(\Qc_{G},\Qc_{G})$. 
	This implies that the class of $\Omega(\Qc_{G_1}, \Qc_{G_2})$ is a new class of reflexive polytopes.
\end{Example}

\end{document}